\documentclass{article}

\usepackage[T2A,T1]{fontenc}
\usepackage[utf8]{inputenc}
\usepackage[russian,english]{babel}

\usepackage{mathptmx}
\usepackage{amssymb,amsmath,amsthm,amsfonts}
\usepackage{graphicx}
\usepackage{bm}
\usepackage[all,cmtip]{xy}
\usepackage{tikz}
\usetikzlibrary{calc, arrows, decorations.markings}
\usepackage{tikz-cd}
\usetikzlibrary{babel}
\usepackage{stmaryrd}
\usepackage{xcolor}
\definecolor{linkblue}{HTML}{003d73}
\definecolor{linkgreen}{HTML}{006161}
\definecolor{linkred}{HTML}{a11950}
\usepackage[pagebackref]{hyperref}
\hypersetup{
	pdftitle={Admissibility and Frame Homotopy for Quaternionic Frames},
	pdfauthor={Tom Needham and Clayton Shonkwiler},
	pdfsubject={frame theory},
	pdfkeywords={frames, frame homotopy, unit-norm tight frames},
	colorlinks=true,
	linkcolor=linkblue,
	citecolor=linkgreen,
	urlcolor=linkred
}
\usepackage{algorithm}
\usepackage{algpseudocode}
\usepackage{pifont}
\usepackage[margin=1in]{geometry}
\usepackage{mathtools}
\usepackage[percent]{overpic}
\usepackage{booktabs}
\usepackage{authblk}
\usepackage{cite}
\usepackage[nameinlink]{cleveref}
\usepackage{quiver}
\usepackage{mathdots}

\crefname{thm}{theorem}{theorems}
\crefname{prop}{proposition}{propositions}
\crefname{cor}{corollary}{corollaries}

\newtheorem{thm}{Theorem}[section]
\newtheorem*{thm*}{Theorem}
\newtheorem{prop}[thm]{Proposition}
\newtheorem{lem}[thm]{Lemma}
\newtheorem{cor}[thm]{Corollary}

\newtheorem*{existence}{Theorem~\ref*{thm:existence}}
\newtheorem*{connectedness}{Theorem~\ref*{thm:connected}}

\theoremstyle{definition}
\newtheorem{definition}[thm]{Definition}
\newtheorem{example}[thm]{Example}

\newcommand{\R}{\mathbb{R}}

\newcommand{\quat}{\mathbb{H}}
\renewcommand{\C}{\mathbb{C}}

\newcommand{\F}{\mathcal{F}}
\newcommand{\frames}{\F^{\quat^d,N}}
\newcommand{\hframes}{\F^{\mathfrak{H},N}}

\newcommand{\Id}{\mathbb{I}}
\newcommand{\tr}{\operatorname{tr}}

\newcommand{\hermitian}{\mathcal{H}}

\renewcommand{\Re}{\operatorname{Re}}
\newcommand{\I}{\mathbf{i}}
\newcommand{\J}{\mathbf{j}}
\newcommand{\K}{\mathbf{k}}

\newcommand{\diag}{\operatorname{diag}}
\newcommand{\conv}{\operatorname{conv}}
\newcommand{\Sp}{\operatorname{Sp}}
\newcommand{\SU}{\operatorname{SU}}

\newcommand{\br}{\boldsymbol{r}}
\newcommand{\blam}{\boldsymbol{\lambda}}

\makeatletter
\newcommand{\subjclass}[2][2010]{%
  \let\@oldtitle\@title%
  \gdef\@title{\@oldtitle\footnotetext{#1 \emph{Mathematics subject classification.} #2}}%
}
\newcommand{\keywords}[1]{%
  \let\@@oldtitle\@title%
  \gdef\@title{\@@oldtitle\footnotetext{\emph{Key words and phrases.} #1.}}%
}
\makeatother

\title{Admissibility and Frame Homotopy for Quaternionic Frames}
\author[$\ast$]{Tom Needham}
\author[$\dag$]{Clayton Shonkwiler}
\affil[$\ast$]{Department of Mathematics, Florida State University, Tallahassee, FL} 
\affil[$\dag$]{Department of Mathematics, Colorado State University, Fort Collins, CO}
\date{}

\keywords{Frame theory, quaternions, isoparametric submanifolds, isotropy representation}
\subjclass{15A29, 42C15, 53C30, 53C40}

\begin{document}

\maketitle

\begin{abstract}
	We consider the following questions: when do there exist quaternionic frames with given frame spectrum and given frame vector norms? When such frames exist, is it always possible to interpolate between any two while fixing their spectra and norms? In other words, the first question is the admissibility question for quaternionic frames and the second is a generalization of the frame homotopy conjecture. We give complete answers to both questions. For the first question, the existence criterion is exactly the same as in the real and complex cases. For the second, the non-empty spaces of quaternionic frames with specified frame spectrum and frame vector norms are always path-connected, just as in the complex case. Our strategy for proving these results is based on interpreting equivalence classes of frames with given frame spectrum as adjoint orbits, which is an approach that is also well-suited to the study of real and complex frames.
\end{abstract}

\section{Introduction}\label{sec:intro}

In a finite-dimensional real or complex Hilbert space $\mathfrak{H}$, an (ordered) \emph{frame} is simply a collection $f_1, \dots , f_N \in \mathfrak{H}$ that spans $\mathfrak{H}$. When $N=d:=\dim \mathfrak{H}$, this is just a basis for $\mathfrak{H}$, but when $N > d$ a frame gives a redundant representation of a signal $v \in \mathfrak{H}$ by
\begin{equation}\label{eq:analysis}
	(\langle v, f_1 \rangle, \dots , \langle v, f_N \rangle)
\end{equation}
which can be more robust to erasures and other corruption of the data than a basis representation~\cite{Casazza:2003vp,Goyal:2001cd,Holmes:2004iv}.

Identifying the frame with the matrix $F = [f_1 | \dots | f_N]$ whose columns are the frame vectors, the redundant representation~\eqref{eq:analysis} corresponds to evaluation of the \emph{analysis operator} $v \mapsto F^\ast v$. Composition of the analysis operator with its adjoint \emph{synthesis operator} $w \mapsto F w$ gives the \emph{frame operator}
\[
	v \mapsto FF^\ast v.
\]
For orthonormal bases---or more generally, and by definition, \emph{Parseval frames}---the frame operator is the identity and the synthesis operator provides a simple method for reconstructing the signal $v$ from the data $F^\ast v$.

As in the case of Parseval frames, it is often desirable to choose frames with a fixed spectrum of the frame operator, for example to provide optimal reconstruction in a given noise model~\cite{Goyal:2001cd,Casazza:2011ev,Viswanath:2002bv}. Likewise, the squared norms $\|f_1\|^2,\dots,\|f_N\|^2$ of the frame vectors are often fixed for both practical and theoretical reasons~\cite{Kovacevic:2007fja,Viswanath:1999hf,Rupf:1994fl,Casazza:2006cx}.

In other words, given vectors $\blam=(\lambda_1, \dots , \lambda_d)$ and $\br=(r_1, \dots , r_N)$ of positive numbers, we are often interested in selecting frames from the space $\hframes_{\blam}(\br)$ of frames $f_1, \dots, f_N \in \mathfrak{H}$ so that $\blam$ is the spectrum of the frame operator $FF^\ast$ and $\|f_i\|^2 = r_i$ for $i=1, \dots , N$. In particular, two natural questions immediately present themselves:
\begin{enumerate}
	\item Does there exist a frame for $\mathfrak{H}$ with prescribed data $\blam$ and $\br$? In other words, is $\hframes_{\blam}(\br)$ non-empty?
	\item Can we interpolate between arbitrary elements of $\hframes_{\blam}(\br)$? In other words, is $\hframes_{\blam}(\br)$ path-connected?
\end{enumerate}

Question 1 has been completely answered by Casazza and Leon~\cite{Casazza:2010ti}, who give a simple compatibility criterion for $\blam$ and $\br$ which determines whether or not $\hframes_{\blam}(\br)$ is empty. Moreover, the same criterion applies in both the real and the complex cases.

Question 2 is a generalization of the well-known \emph{frame homotopy conjecture}, which was posed by Larson in a 2002 REU and first appeared in the literature in Dykema and Strawn's 2006 paper~\cite{Dykema:2006ux}. This conjecture says that when $\blam$ and $\br$ are constant vectors, the space $\hframes_{\blam}(\br)$ is path-connected (when considered with the natural subspace topology). Up to scale one can assume that the constant vector $\br = (1, \dots , 1)$, and hence the frames under consideration are \emph{unit-norm tight frames}, so the frame homotopy conjecture says that the space of unit-norm tight frames in $\mathfrak{H}$ is path-connected.

The frame homotopy conjecture was proved for both real and complex frames by Cahill, Mixon, and Strawn in a 2017 paper~\cite{Cahill:2017gv}. In previous work~\cite{NeedhamSGC}, we gave a complete answer to Question 2 for complex frames, showing that the space $\F^{\C^d,N}_{\blam}(\br)$ is always path-connected. For real frames with nonconstant $\blam$ or $\br$, Question 2 remains open.

There has been a recent flourishing of interest in \emph{quaternionic frames}~\cite{iverson_note_2021,Waldron:2020tp,Cohn:2016bz,EtTaoui:2020kf,Waldron:2020ti,Khokulan:2017ic,Sharma:2019js,Virender:2020ua}; that is, frames in $\quat^d$ for $d \geq 1$ and $\quat$ being the 4-dimensional skew field of quaternions. Since the group $\Sp(1)$ of unit quaternions is isomorphic to $\operatorname{Spin}(3)$ (the universal cover of the rotation group $\operatorname{SO}(3)$), the quaternions can be interpreted as a cone over $\operatorname{Spin}(3)$, so they are well-suited to parameterizing rotations in $\R^3$~\cite{Hanson:2006tr}. Just as complex numbers consist of a magnitude and a phase, quaternions consist of a magnitude and a versor, which determines a particular 3D rotation. Consequently, vectors $v \in \quat^d$ can be used to record both magnitude and orientation information, for example in framed curves like polymers or inflatable elastic rods~\cite{Cantarella:2013bla,Hanson:2012vb,Howard:2011fj,Needham:2017vd}. Moreover, quaternions have become an increasingly popular tool for representing signals in data-driven applications. For example, quaternion-valued signals have recently been used to encode RGB images \cite{ell2006hypercomplex,fletcher2017development}, measurements in industrial machinery \cite{yi2017quaternion}, and multicomponent seismic measurements \cite{zhao2020quaternion}; see also the special issue of \emph{Signal Processing} on Hypercomplex Signal Processing~\cite{SPspecialissue}. This motivates the extension of signal processing techniques for classical (i.e., real- or complex-valued) signals to quaternionic signals.

Despite the difficulties introduced by the non-commutativity of the quaternions, it is still possible to define frames in $\quat^d$, the norms of the frame vectors, and the spectrum of the frame operator (see \Cref{sub:quaternionic_frames}). In other words, the space $\frames_{\blam}(\br)$ of length-$N$ frames in $\quat^d$ with fixed frame spectrum $\blam$ and fixed frame vector norms $\|f_i\|^2 = r_i$ exists, and the goal of this paper is to give complete answers to Questions 1 and 2 in this case. Specifically, we show:

\begin{thm}\label{thm:existence}
	Let $N$, $d$, $\blam$, and $\br$ be as above, and additionally assume that $\blam$ and $\br$ are sorted in non-increasing order: $\lambda_1 \geq \dots \geq \lambda_d > 0$ and $r_1 \geq \dots \geq r_N > 0$. Then $\frames_{\blam}(\br)$ is non-empty if and only if
	\[
		\sum_{i=1}^k r_i \leq \sum_{i=1}^k \lambda_i \text{ for all } k=1, \dots , d \qquad \text{and} \qquad \sum_{i=1}^N r_i  = \sum_{i=1}^d \lambda_i.
	\]
\end{thm}

\begin{thm}\label{thm:connected}
	For any $\br$ and $\blam$, the space $\frames_{\blam}(\br)$ is path-connected.
\end{thm}

Since the empty set is trivially path-connected, the substance of \Cref{thm:connected} is that all of the non-empty $\frames_{\blam}(\br)$ spaces are path-connected.

Note that the condition in \Cref{thm:existence} is exactly the same as that given by Casazza and Leon in the real and complex cases~\cite{Casazza:2010ti}. In other words, for given $\blam$ and $\br$ either all three of $\F^{\R^d,N}_{\blam}(\br)$, $\F^{\C^d,N}_{\blam}(\br)$, and $\frames_{\blam}(\br)$ are empty, or all three are non-empty. Likewise, the result in \Cref{thm:connected} is the same as in the complex case~\cite{NeedhamSGC}, so that $\F^{\C^d,N}_{\blam}(\br)$ and $\frames_{\blam}(\br)$ are both path-connected. In particular, \Cref{thm:existence,thm:connected} imply that all spaces of quaternionic unit-norm tight frames are non-empty and path-connected, so the standard frame homotopy conjecture is true in the quaternionic setting.

Our strategy for proving the complex version of \Cref{thm:connected} in~\cite{NeedhamSGC} was based on symplectic geometry, which is extremely well-suited to the complex setting, but not to either the real or quaternionic settings. In this paper, we instead identify (equivalence classes of) frames with fixed frame spectrum $\blam$ with an adjoint orbit of the action of the symplectic group\footnote{The (compact) symplectic group is the quaternionic analog of the unitary group; see \Cref{sub:quaternions} for the definition.} $\Sp(N)$ on the space of $N \times N$ Hermitian matrices, which is the natural home of the Gram matrix $F^\ast F$ associated to a frame $F$. This is not special to quaternionic frames: one can make an analogous identification for real and complex frames, so we view this as a unifying perspective on all three classes of frames.

The advantage is that such adjoint orbits are extremely nice geometrically. In particular, they are motivating special cases of both Kostant's convexity theorem~\cite{Kostant:1973ff} and of the theory of isoparametric submanifolds~\cite[Chapter~6]{Palais:1988ks}, and the basic strategy for proving \Cref{thm:existence,thm:connected} is to apply those very general tools to these specific problems.

Specifically, since $\blam$ and $\br$ are essentially the spectrum and diagonal entries, respectively, of the Gram matrix $F^\ast F$, the real and complex versions of \Cref{thm:existence} can be proved using the Schur--Horn theorem~\cite{Antezana:2007ci,Tropp:2005do}. Likewise, we will prove \Cref{thm:existence} using a quaternionic Schur--Horn theorem. This quaternionic version of Schur--Horn is an easy consequence of Kostant's convexity theorem, but its statement does not seem to be readily accessible in the literature (though see~\cite[Example~8.6]{kobert_spectrahedral_2021}). Since it may be of some independent interest, we give a statement and proof in \Cref{thm:quaternionic Schur-Horn}. On the other hand, adjoint orbits (and, more generally, isoparametric submanifolds) share many of the nice features of Hamiltonian manifolds~\cite{Terng:1986hg,Mare:2005eo}. In particular, in this setting we have access to a version of Atiyah's connectedness theorem~\cite{Atiyah:1982ih} (see \Cref{thm:Mare}), and with that tool in hand the proof of \Cref{thm:connected} goes much as it did in the complex case.

\section{Quaternions and Quaternionic Frames} 
\label{sec:background}
\subsection{Quaternions} 
\label{sub:quaternions}
In thinking about $\quat^m$ as a vector space over $\quat$, some care is required. First, recall that $\quat$ is a skew field whose elements can be written as 
\[
	a + b \I + c \J + d \K
\]
with multiplication given by the identities
\[
	\I^2 = \J^2 = \K^2 = \I \J \K = -1,
\]
which in particular implies $\I \J = \K$, $\J \I = -\K$, etc. The products of distinct elements of $\{\I,\J,\K\}$ are encoded in the diagram
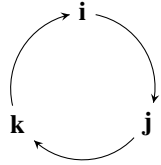
\[
	\begin{tikzpicture}[decoration = {markings,
	    mark = at position 0.999 with {\arrow[>=stealth]{>}}
	  }
	  ]
	  \path (90 :1cm) node (i) {$\I$}
	        (330:1cm) node (j) {$\J$}
	        (210:1cm) node (k) {$\K$};
	  \draw[postaction = decorate] (i) to[bend left=45] (j);
	  \draw[postaction = decorate] (j) to[bend left=45] (k);
	  \draw[postaction = decorate] (k) to[bend left=45] (i);
	\end{tikzpicture}
\]
which for example says that $\K \I = \J$ since the product follows arrows with their correct orientation and $\K \J = -\I$ since this product reverses directions of arrows.

Instead of thinking in terms of quadruples of real numbers, it is also often convenient to interpret quaternions as pairs of complex numbers:
\[
	q = a+b\I+c\J+d\K = (a+b\I) + (c+d\I)\J.
\]
This makes it easy to identify an element of $\quat$ with a $2 \times 2$ complex matrix:
\begin{equation}\label{eq:H to C map}
	z + w\J \mapsto \begin{bmatrix} z & w \\ -\overline{w} & \overline{z} \end{bmatrix}.
\end{equation}
This map is an algebra isomorphism between $\quat$ and the collection of $2 \times 2$ complex matrices of the given form.

The \emph{conjugate} of a quaternion $q =a + b\I+c\J+d \K \in \quat$ is defined by
\[
	\overline{q} = a - b\I-c\J-d\K,
\]
and the \emph{modulus} by
\[
	|q| = \sqrt{q \overline{q}} = \sqrt{a^2+b^2+c^2+d^2}.
\]

Since quaternionic multiplication is non-commutative, we have to distinguish between left and right vector spaces over $\quat$. In this paper we will follow Waldron's conventions~\cite{Waldron:2020ti} and consider $\quat^m$ as a right vector space over $\quat$, meaning that scalar multiplication happens on the right. This is done so that left matrix multiplication is linear: for $v_1,\dots , v_N \in \quat^d$ thought of as column vectors, $\alpha_1,\dots , \alpha_N \in \quat$, and a matrix $A \in \quat^{m \times d}$, we have
\[
	A(v_1 \alpha_1 + \dots + v_N \alpha_N) = (A v_1) \alpha_1 + \dots + (A v_N) \alpha_N.
\]

Much of linear algebra over $\quat$ goes through as over $\R$ or $\C$; see~\cite{Zhang:1997cd,Rodman:2014ue} for more, but in particular recall that $\quat^m$ has a standard $\quat$-valued Hermitian inner product given by
\begin{equation}\label{eq:inner product}
	\langle v, w \rangle = \sum_i \overline{w}_j v_j \in \quat
\end{equation}
for all $v,w \in \quat^m$. The \emph{Frobenius inner product} on matrices $A,B \in \quat^{m \times k}$ is given by
\[
	\langle A, B \rangle_F := \tr(B^\ast A),
\]
where, as usual, $B^\ast$ is the conjugate transpose of the matrix $B$. This agrees with the Hermitian inner product~\eqref{eq:inner product} on vectorized versions of the matrices.

The quaternionic analog of the unitary group is usually called the \emph{symplectic group} (sometimes the \emph{compact symplectic group}) and denoted $\Sp(m)$. That is, a $m \times m$ quaternionic matrix $U \in \quat^{m \times m}$ is symplectic if
\[
	\langle U v, U w \rangle = \langle v , w \rangle
\]
for all $v,w \in \quat^m$; equivalently, $U^\ast U = \Id_m$, the $m \times m$ identity matrix. $\Sp(m)$ is a compact, semisimple Lie group with type-$C$ Lie algebra $\mathfrak{sp}(m)$ consisting of the skew-Hermitian $m \times m$ quaternionic matrices; that is, those matrices $A$ so that $A^\ast = -A$. A simple parameter count shows that
\[
	\dim_\R \Sp(m) = \dim_\R \mathfrak{sp}(m) = 3m+4\frac{m(m-1)}{2} = 2m^2+m.
\]

Finally, we introduce the notation
\[
	\hermitian(m) = \{A \in \quat^{m \times m} : A^\ast = A\}
\]
for the $(2m^2-m)$-dimensional real vector space of $m \times m$ Hermitian quaternionic matrices, and let $\hermitian_0(m)$ be the subspace of traceless Hermitian matrices.

\subsection{Quaternionic Frames} 
\label{sub:quaternionic_frames}
Frames in $\quat^d$ are defined in the usual way~\cite{Waldron:2020ti,Khokulan:2017ic,Sharma:2019js,Virender:2020ua}: a sequence $(f_i)$ of vectors in $\quat^d$ is a \emph{frame} if there exist $0<A\leq B < \infty$ so that
\begin{equation}\label{eq:frame inequality}
	A\|v\|^2 \leq \sum_i |\langle v, f_i\rangle|^2 \leq B \|v\|^2.
\end{equation}
For finite collections of vectors the upper bound is automatically satisfied with $B = \sum_i \|f_i\|^2$ and the lower bound is satisfied when the $f_i$ span $\quat^d$. In other words, a finite collection of vectors $f_1, \dots , f_N \in \quat^d$ is a frame for $\quat^d$ if and only if $\{f_1, \dots , f_N\}$ is a spanning set for $\quat^d$. We will use $\frames$ to denote the collection of all frames consisting of $N$ vectors in $\quat^d$. 

A frame is called \emph{tight} if we can choose $A=B$ in \eqref{eq:frame inequality}, and \emph{Parseval} if $A=B=1$. We get a nice alternative characterization of these frames by introducing some operators. To do so, we will usually identify a frame $f_1, \dots , f_N$ with the $d \times N$ matrix $F = [f_1 | \dots | f_N ] \in \quat^{d \times N}$ whose columns are the frame vectors, and we will often just write $F \in \frames$. The \emph{analysis operator} associated to the frame is the map $\quat^d \to \quat^N$ given by
\[
	v \mapsto (\langle v, f_1 \rangle , \dots , \langle v, f_N\rangle ) = F^\ast v,
\]
and the \emph{synthesis operator} is the map $\quat^N \to \quat^d$ given by
\[
	w \mapsto \sum_{i=1}^N f_i w_i = F w.
\]

Composing these two operators one way gives the \emph{frame operator} $S:\quat^d \to \quat^d$ given by
\[
	S(v) = \sum_{i=1}^N f_i \langle v, f_i \rangle = FF^\ast v.
\]
In this interpretation, a tight frame has frame operator $S = FF^\ast = A \Id_d$, so in particular a Parseval frame satisfies $FF^\ast = \Id_d$. Composing in the other way gives the \emph{Gram matrix} $F^\ast F$ whose entries are the pairwise inner products of the frame vectors. For real and complex frames, cyclic invariance of the trace immediately implies that $\tr(FF^\ast) = \tr(F^\ast F)$, which is an important and frequently-used identity. Over the quaternions, the trace of a product of matrices is not generally invariant under cyclic permutations of terms, but the real part is: $\Re(\tr(AB)) = \Re(\tr(BA))$ for quaternionic matrices $A$ and $B$ of dimensions for which these products make sense. Since the frame operator and Gram matrix are both Hermitian, and hence have real eigenvalues and real trace, it follows that
\begin{equation}\label{eq:cyclic trace}
	\tr(FF^\ast) = \Re(\tr(FF^\ast)) = \Re(\tr(F^\ast F)) = \tr(F^\ast F) = \sum_{i=1}^N \|f_i\|^2.
\end{equation}

In fact, since there is a singular value decomposition for quaternionic matrices~\cite{Zhang:1997cd,Rodman:2014ue}, the usual argument from the real and complex cases shows that the frame operator $FF^\ast \in \hermitian(d)$ and the Gram matrix $F^\ast F \in \hermitian(N)$ have the same nonzero (right) eigenvalues. In particular:

\begin{prop}\label{prop:spectrum}
	If $F \in \frames$ has frame operator $FF^\ast$ with spectrum $\blam = (\lambda_1, \dots , \lambda_d) \in \R_+^d$,\footnote{Since the columns of $F$ are a spanning set for $\quat^d$, the frame operator $FF^\ast$ must be nonsingular, so we know $\lambda_i > 0$ for all $i=1,\dots, d$.} then the Gram matrix $F^\ast F$ has spectrum $\widetilde{\blam} := (\lambda_1, \dots , \lambda_d, 0 , \dots , 0)$.
\end{prop}

Since we are interested in frames with specified spectrum of the frame operator, we introduce some terminology and notation:

\begin{definition}\label{def:frame spectrum}
	If $F \in \frames$, call the spectrum of $FF^\ast$ the \emph{frame spectrum} of $F$. For given $\blam = (\lambda_1, \dots , \lambda_d) \in \R_+^d$, let $\frames_{\blam}$ be the set of all frames in $\frames$ with frame spectrum equal to $\blam$.
\end{definition}

Since Gram matrices are positive semidefinite and the spectral decomposition of Hermitian quaternionic matrices works as expected~\cite{Zhang:1997cd,Rodman:2014ue}, the same proof as in the real or complex case shows that every positive semidefinite Hermitian quaternionic matrix is the Gram matrix of a collection of vectors which is a frame for its span, and that this collection is uniquely determined up to the action of the symplectic group. In other words, if $M \in \hermitian(N)$ is positive semidefinite of rank $d$, then there exists $F = [f_1 | \dots | f_N] \in \frames \subset \quat^{d \times N}$ so that $M = F^\ast F$. Notice that, for $U \in \Sp(d)$, the frame $UF$ has the same Gram matrix: $(UF)^\ast (UF) = F^\ast U^\ast U F = F^\ast F = M$. Conversely, if $G \in \quat^{d \times N}$ so that $G^\ast G = M$, then $G = UF$ for some $U \in \Sp(d)$. 

We summarize the above discussion in the following proposition:

\begin{prop}\label{prop:equivalence classes}
	$\Sp(d)$-equivalence classes of frames in $\frames$ are uniquely determined by their Gram matrices, which consist of all rank-$d$ positive semidefinite elements of $\hermitian(N)$.
\end{prop}

\section{Adjoint Orbits, Convexity, and Isoparametric Submanifolds} 
\label{sec:isoparametric}

Our next goal is to connect the story of Gram matrices to adjoint orbits and isoparametric submanifolds. We will largely follow Mare's discussion~\cite[Example~5.4]{Mare:2005eo}; see \cite[\S III.7]{Helgason:1978vb} for a more general exposition of Cartan decompositions.

Thinking of a square quaternionic matrix $A \in \quat^{m \times m}$ as
\[
	A = Z + W \J
\]
for $Z,W \in \C^{m \times m}$, we can define a map $\Psi_m:\quat^{m \times m} \to \C^{2m \times 2m}$ analogously to~\eqref{eq:H to C map}:
\[
	\Psi_m(Z+W\J) :=  \begin{bmatrix} Z & W \\ -\overline{W} & \overline{Z} \end{bmatrix}.
\]

If $A \in \quat^{m \times m}$ is invertible, then $\Psi_m(A)$ is as well. Moreover, the image $\Psi_m(\Sp(m))$ of $\Sp(m)$ under this map is a subgroup of $\SU(2m)$, the group of $2m \times 2m$ unitary matrices with determinant 1, and in fact $\Psi_m(\Sp(m))$ is precisely the fixed point set inside $\SU(2m)$ of the involution $\sigma:\C^{2m \times 2m} \to \C^{2m \times 2m}$ given by
\[
	\sigma(M) = \Omega^\ast \overline{M} \Omega,
\]
where $\Omega = \begin{bmatrix} 0 & \Id_m \\ -\Id_m & 0 \end{bmatrix}$. 

Passing to the Lie algebra gives the Cartan decomposition
\[
	\mathfrak{su}(2m) = \mathfrak{k} \oplus \mathfrak{p},
\]
where $\mathfrak{k}$ is the $(+1)$-eigenspace of the linearization $d_{\Id_m}\sigma: \mathfrak{su}(2m) \to \mathfrak{su}(2m)$ and $\mathfrak{p}$ is the $(-1)$-eigenspace.\footnote{Since $d_{\Id_m}\sigma(A) = -A^\ast$ is an involution, its only eigenvalues are $\pm 1$.} The subalgebra $\mathfrak{k}$ is just the Lie algebra of $\Psi_m(\Sp(m)) \simeq \Sp(m)$, so $\mathfrak{k} \simeq \mathfrak{sp}(m)$, with an explicit isomorphism given by $\Psi_m|_{\mathfrak{sp}(m)}: \mathfrak{sp}(m) \to \mathfrak{k}$. 

Notice that $[\mathfrak{k},\mathfrak{p}] \subseteq \mathfrak{p}$, and so the adjoint action of $\Psi_m(\Sp(m))$ on $\mathfrak{su}(2m)$ restricts to an action on $\mathfrak{p}$. At first glance this appears slightly esoteric, but we can use $\Psi_m$ to relate $\mathfrak{p}$ to a more familiar collection of matrices. Specifically, $\Psi_m|_{\hermitian_0(m)}$ gives an $\Sp(m)$-equivariant linear isomorphism between the traceless Hermitian matrices $\hermitian_0(m)$ and $\mathfrak{p}$, and the adjoint action of $\Psi_m(\Sp(m))$ on $\mathfrak{p}$ corresponds to the conjugation action of $\Sp(m)$ on $\hermitian_0(m)$.

Under this isomorphism, we can identify the standard maximal abelian subspace of $\mathfrak{p}$ with the set $\mathfrak{a} \subset \hermitian_0(m)$ of real diagonal $m \times m$ matrices with trace $0$. The orbit $\mathcal{O}_\Lambda := \Sp(m)\cdot \Lambda$ of a point $\Lambda = \diag(\lambda_1, \dots , \lambda_m) \in \mathfrak{a}$ is simply
\[
	\mathcal{O}_\Lambda = \{U \Lambda U^\ast : U \in \Sp(m)\},
\]
which is precisely the collection of Hermitian $m \times m$ quaternionic matrices with spectrum $(\lambda_1, \dots , \lambda_m)$. At this point, it is hopefully clear that we intend to use \Cref{prop:spectrum,prop:equivalence classes} to relate $\Sp(d)$-equivalence classes of frames to orbits of this form. Before doing so, let's see what desirable features these orbits have.

Needless to say, we could have introduced these orbits without recourse to Cartan decompositions or adjoint actions. The point of approaching things in this slightly roundabout way, though, is that we can now easily see that $\mathcal{O}_\Lambda$ fits into both \emph{Kostant's convexity theorem} and the story of so-called \emph{isoparametric submanifolds}.

\subsection{Kostant's Convexity Theorem} 
\label{sub:Kostant}

Consider the Cartan decomposition
\[
	\mathfrak{G} = \mathfrak{K} \oplus \mathfrak{P}
\]
of the Lie algebra $\mathfrak{G}$ of a semisimple Lie group $G$.\footnote{In defiance of the usual convention, we are using capital fraktur letters here for the Lie algebra and its subspaces, so as not to confuse the general $\mathfrak{K}$ and $\mathfrak{P}$ discussed here with the specific $\mathfrak{k}$ and $\mathfrak{p}$ defined above.} Then, as above, $\mathfrak{K}$ is the Lie algebra of a compact Lie subgroup $K \subset G$, and the adjoint action of $K$ on $\mathfrak{G}$ restricts to an action on $\mathfrak{P}$. Let $\mathfrak{A} \subset \mathfrak{P}$ be a maximal abelian subspace and let $P:\mathfrak{P} \to \mathfrak{A}$ be the orthogonal (with respect to the Killing form) projection. The \emph{Weyl group} $W$ associated with the pair $(\mathfrak{A},\mathfrak{G})$ is the finite group $N_K(\mathfrak{A})/Z_K(\mathfrak{A})$, where $N_K(\mathfrak{A})$ is the normalizer of $\mathfrak{A}$ in $K$, and $Z_K(\mathfrak{A})$ is its centralizer.

Now, let $a \in \mathfrak{A} \subset \mathfrak{P}$ and let $\mathcal{O}_a$ be the orbit of $a$ under the adjoint action of $K$ on $\mathfrak{P}$. Then Kostant's convexity theorem characterizes the image of $\mathcal{O}_a$ under $P$:

\begin{thm}[Kostant~\cite{Kostant:1973ff}]\label{thm:Kostant}
	$P(\mathcal{O}_a) = \conv(W\cdot a)$, the convex hull of the Weyl orbit of $a \in \mathfrak{A}$.
\end{thm}

In the case where $G=\operatorname{SL}_n(\C)$ and $K = \SU(n)$, this is essentially the Schur--Horn theorem~\cite{schur1923uber,horn1954doubly}, which says that the diagonal entries of Hermitian matrices with fixed spectrum fill out the convex hull of the collection of vectors given by all possible re-orderings of the spectrum.

With $G=\SU(2m)$, $K=\Sp(m)$, and the Cartan decomposition $\mathfrak{su}(2m) = \mathfrak{k} \oplus \mathfrak{p}$ as above, \Cref{thm:Kostant} implies the following quaternionic analog of the Schur--Horn theorem (compare to~\cite[Example~8.6]{kobert_spectrahedral_2021}):

\begin{thm}\label{thm:quaternionic Schur-Horn}
	Let $\blam= (\lambda_1 , \dots , \lambda_m) \in \R^m$. Let $\hermitian_{\blam}(m)$ be the collection of quaternionic Hermitian $m \times m$ matrices with spectrum $\blam$. Let $\Delta: \hermitian(m) \to \R^m$ record the diagonal entries of a matrix. Then
	\[
		\Delta(\hermitian_{\blam}(m)) = \conv(S_m \cdot \blam),
	\]
	the convex hull of the permutation orbit of $\blam$.
\end{thm}

\begin{proof}
	First of all, if $\lambda_1 + \dots + \lambda_m = 0$, then $\hermitian_{\blam}(m) \subset \hermitian_0(m)$ corresponds to the orbit of $\Lambda = \diag(\lambda_1, \dots , \lambda_m)$ under the adjoint action of $\Sp(m)$ on $\mathfrak{p} \simeq \hermitian_0(m)$. Moreover, the Weyl group associated to $(\mathfrak{a},\mathfrak{su}(2m))$ is the symmetric group $S_m$ and the orthogonal projection $\mathfrak{p} \to \mathfrak{a}$ corresponds to the diagonal entry map $\Delta$, so \Cref{thm:Kostant} says exactly that
	\[
		\Delta(\hermitian_{\blam}(m)) = \conv(S_m \cdot \blam).
	\]
	
	If $\lambda_1 + \dots + \lambda_m \neq 0$, then we just need to re-center, which we do using the map $\tau: \hermitian_{\blam}(m) \to \mathcal{O}_\Lambda$ defined in \Cref{lem:re-center} below. In the notation of that lemma, \Cref{thm:Kostant} implies that 
	\[
		(\Delta \circ \tau)(\hermitian_{\blam}(m)) = \conv(t(S_m \cdot \blam)).
	\] 
	
	Since taking the convex hull is $t$-equivariant and since $\Delta \circ \tau = t \circ \Delta$, it follows that $\Delta(\hermitian_{\blam}(m)) = \conv(S_m \cdot \blam)$, as desired.
\end{proof}

\begin{lem}\label{lem:re-center}
	Let $\blam = (\lambda_1, \dots , \lambda_m) \in \R^m$, let $\sigma = \lambda_1 + \dots + \lambda_m$, and define $\Lambda := \diag(\lambda_1, \dots , \lambda_m)-\frac{\sigma}{m} \Id_m$. Then the map $B \mapsto B - \frac{\sigma}{m} \Id_m$ defines a diffeomorphism $\tau: \hermitian_{\blam}(m) \to \mathcal{O}_\Lambda$ which is equivariant with respect to the $\Sp(m)$ actions on domain and range and makes the following diagram commute:
	\[\begin{tikzcd}
		{\hermitian_{\blam}(m)} && {\mathcal{O}_\Lambda} \\
		\\
		{\R^m} && {\R^m}
		\arrow["\tau", from=1-1, to=1-3]
		\arrow["\Delta"', from=1-1, to=3-1] 
		\arrow["t", from=3-1, to=3-3]
		\arrow["\Delta", from=1-3, to=3-3]
	\end{tikzcd}\]
	Here $t: \R^m \to \R^m$ is the translation map $t(\boldsymbol{x}) := \boldsymbol{x}-\left(\frac{\sigma}{m}, \dots , \frac{\sigma}{m}\right)$.
\end{lem}

\begin{proof}
	First, $\tau$ is certainly well-defined and smooth with smooth inverse, so it is a diffeomorphism. 
	
	If $B \in \hermitian_{\blam}(m)$ and $U \in \Sp(d)$, then
	\[
		U \tau(B) U^\ast = U\left(B - \frac{\sigma}{m}\Id_m\right)U^\ast = UBU^\ast - \frac{\sigma}{m}\Id_m = \tau(UBU^\ast),
	\]
	so $\tau$ is $\Sp(m)$-equivariant.
	
	Finally, for $B \in \hermitian_{\blam}(m)$, 
	\[
		(\Delta \circ \tau)(B) = \Delta\left(B - \frac{\sigma}{m}\Id_m\right) = \Delta(B) - \left(\frac{\sigma}{m},\dots , \frac{\sigma}{m}\right) = (t \circ \Delta)(B),
	\]
	so the diagram commutes.
\end{proof}

\subsection{Isoparametric Submanifolds} 
\label{sub:isoparametric submanifolds}

Isoparametric hypersurfaces were first studied by Cartan~\cite{Cartan:1938kh,Cartan:1939gv,Cartan:1939wj,Cartan:tw,Nomizu:1975bq}; for a more modern and general introduction see \cite[Chapter~6]{Palais:1988ks} or~\cite{Terng:1985eu}. First, we give the definition, the details of which will not concern us greatly:

\begin{definition}\label{def:isoparametric}
	A submanifold $M$ of a Riemannian manifold $N$ is \emph{isoparametric} if its normal bundle $\nu(M)$ is flat and the principal curvatures along any parallel normal field of $M$ are constant.
\end{definition}

Principal orbits of isotropy representations are classic examples of isoparametric submanifolds. In the context of our story, we have:

\begin{prop}[{see, e.g.,~\cite[Example~6.5.6]{Palais:1988ks}}]\label{prop:orbits are isoparametric}
	If $\Lambda \in \mathfrak{a}$ is generic (i.e., $\lambda_i \neq \lambda_j$ for all $i \neq j$), then $\mathcal{O}_\Lambda$ is an isoparametric submanifold of $\hermitian_0(N)$, thought of as a Riemannian manifold by taking the Frobenius inner product on each tangent space. Moreover, the normal space to $\mathcal{O}_\Lambda$ at $\Lambda$ is just $\mathfrak{a}$. 
	
	If $\Lambda$ is not generic, then $\mathcal{O}_\Lambda$ is parallel to an isoparametric submanifold of $\hermitian_0(N)$.
\end{prop}

To explain the terminology in the previous sentence, if $M \subset \R^n$ is isoparametric and $\xi$ is a parallel section of the normal bundle $\nu(M)$, then $M_\xi := \{p + \xi(p) : p \in M\}$ is a \emph{parallel submanifold} to $M$, and these parallel submanifolds give a singular foliation of the ambient space. In the special case of isotropy orbits, the orbit foliation and the parallel foliation coincide.

Isoparametric submanifolds and their parallels have some of the nice features of symplectic manifolds admitting Hamiltonian torus actions without necessarily being symplectic. For example, the following fundamental result in symplectic geometry has an analog in the isoparametric setting.

\begin{thm}[{Atiyah~\cite{Atiyah:1982ih} and Guillemin--Sternberg~\cite{Guillemin:1982gx}}]\label{thm:AGS}
	Let $(M,\omega)$ be a compact symplectic manifold admitting a Hamiltonian action of a torus $T\simeq \operatorname{U}(1)^n$. Let $\mu:M \to \R^n$ be the associated momentum map. 
	\begin{itemize}
		\item For any $v \in \R^n$, $\mu^{-1}(v)$ is either empty or connected.
		\item $\mu(M)$ is the convex hull of the images of the fixed points of the $T$-action.
	\end{itemize}
\end{thm}

\begin{example}
	Let $K$ be a compact Lie group of rank $n$ with Lie algebra $\mathfrak{K}$. Let $\mathcal{O} \subset \mathfrak{K}^\ast$ be an orbit of the coadjoint action of $K$ on $\mathfrak{K}^\ast$. This action of $K$ on $\mathcal{O}$ is Hamiltonian with momentum map being the inclusion $\mathcal{O} \hookrightarrow \mathfrak{K}^\ast$~\cite[Example~5.3.11]{mcduff2017introduction}. 
	
	If $T \subset K$ is a maximal torus, then $T \simeq U(1)^n$ and the coadjoint action of $T$ on $\mathcal{O}$ is also Hamiltonian, with moment map given by the restriction of the projection $\mathfrak{K}^\ast \to \mathfrak{T}^\ast \simeq \R^n$ induced by the inclusion $T \hookrightarrow K$~\cite[Proposition~II.1.10]{Audin:2004bh}. Then \Cref{thm:AGS} implies that the image of this map is convex and its non-empty level sets are connected.
	
	In this case, convexity actually follows from \Cref{thm:Kostant}: let $G = K_\C$ be the complexification of $K$, with Lie algebra $\mathfrak{G} = \mathfrak{K} \oplus \I \mathfrak{K}$. This is a Cartan decomposition and the dual Lie algebra $\mathfrak{K}^\ast$ can be identified with the complementary subspace $\I \mathfrak{K}$ so that the coadjoint action of $K$ on $\mathfrak{K}^\ast$ corresponds to the adjoint action of $K$ on $\I \mathfrak{K}$ and $\mathfrak{T}^\ast \subset \mathfrak{K}^\ast$ corresponds to a maximal abelian subalgebra $\mathfrak{A} \subset \I \mathfrak{K}$. Hence, the projection map that \Cref{thm:Kostant} says has convex image is just the momentum map of the torus action. Of course, Kostant's convexity theorem was a major inspiration for Atiyah and Guillemin--Sternberg, whose result can be interpreted as generalizing this case of Kostant's theorem to arbitrary Hamiltonian torus actions. 
	
	On the other hand, if $\mathcal{O}$ is a principal orbit, then $\mathcal{O}$ is an isoparametric submanifold of $\mathfrak{K}^\ast$ and the normal space at a point can be identified with $\mathfrak{A} \simeq \mathfrak{T}^\ast$. This gives yet another interpretation of the momentum map as the orthogonal projection onto the normal space at a point. Terng showed that, under this interpretation, the convexity part of \Cref{thm:AGS} generalizes to arbitrary isoparametric submanifolds:
\end{example}

\begin{thm}[{Terng~\cite{Terng:1986hg}}]\label{thm:Terng convexity}
	Let $M_\xi \subset \R^{n}$ be parallel to an isoparametric submanifold $M$. Let $p \in M$ so that $p + \xi(p) \in M_\xi$, let $\nu_p(M)$ be the normal space to $M$ at $p$, and let $P:M_\xi \to \nu_p(M)$ be orthogonal projection. Then $P(M_\xi)$ is a convex polytope.
\end{thm}

Given this, the projection map $P$ is like a momentum map for isoparametric submanifolds and their parallels, even though these submanifolds need not be symplectic. It is then reasonable to ask whether, as in the first part of \Cref{thm:AGS}, the level sets of $P$ are connected. They needn't be in general, but Mare gave a sufficient condition for all non-empty level sets of $P$ to be connected:

\begin{thm}[{Mare~\cite[Theorem~1.2 and Remark~1.3(a)]{Mare:2005eo}}]\label{thm:Mare}
	Let $M \subset \R^n$ be an isoparametric submanifold with all multiplicities $\geq 2$ and let $M_\xi$ be parallel to $M$. If $p \in M$ and $b \in \nu_p(M)$ is in the image of the projection $P: M_\xi \to \nu_p(M)$, then $P^{-1}(b)$ is connected.
\end{thm}

See Mare's paper for a general definition of multiplicities; in the setting of \Cref{thm:Kostant}, where we are considering orbits of the adjoint action of $K \subset G$ on the complementary subspace $\mathfrak{P} \subset \mathfrak{G}$, the multiplicities are the differences in dimension between a principal orbit and subprincipal orbits~\cite{Palais:1988ks,Hsiang:1988ka}. Recall from \Cref{prop:orbits are isoparametric} that the normal space can be identified with the maximal abelian subspace $\mathfrak{A}$.
 
\section{Existence} 
\label{sec:admissibility}

Let $N \geq d \geq 1$ be integers and let $\blam=(\lambda_1, \dots, \lambda_d)$ and $\br = (r_1, \dots , r_N)$ be lists of positive numbers. Our main object of interest is the collection $\frames_{\blam}(\br)$ of frames $F =[f_1 | \dots | f_N] \in \frames \subset \quat^{d \times N}$ so that $FF^\ast$ has spectrum $\blam$ and $\|f_i\|^2 = r_i$ for $i=1, \dots , N$. Our first goal is to answer the question: are there any such frames?

The analogous question for real and complex frames was answered by Casazza and Leon~\cite{Casazza:2010ti}, and the answer in the quaternionic case is essentially the same:

\begin{thm}\label{thm:admissibility}
	Let $N$, $d$, $\blam$, and $\br$ be as above. The space $\frames_{\blam}(\br)$ is non-empty if and only if $\br \in \conv(S_N \cdot \widetilde{\blam})$, the convex hull of the collection of vectors in $\R^N$ given by permuting the entries of $\widetilde{\blam} := (\lambda_1, \dots , \lambda_d,0,\dots,0)$.
\end{thm}

\begin{proof}
	\Cref{thm:quaternionic Schur-Horn} will be the key to the proof, so the goal is to relate $\frames_{\blam}(\br)$ to a space of the form $\hermitian_{\boldsymbol{\eta}}(m)$ for some choice of $\boldsymbol{\eta}$ and $m$. 

	Recall from \Cref{prop:equivalence classes} that $\Sp(d)$-equivalence classes of frames in $\frames$ are determined by their Gram matrices. Specifically, the space $\frames_{\blam}/\Sp(d)$ of $\Sp(d)$-equivalence classes of frames with frame spectrum $\blam$ is diffeomorphic to the space of all possible Gram matrices of frames in $\frames_{\blam}$.

	Now, if $F \in \frames_{\blam}$, then definitionally $FF^\ast$ has spectrum $\blam$, and hence the Gram matrix $F^\ast F$ has spectrum $\widetilde{\blam}=(\lambda_1, \dots , \lambda_d, 0, \dots , 0) \in \R^N$. Conversely, any matrix in $\hermitian_{\widetilde{\blam}}(N)$ can be realized as the Gram matrix of a frame in $\frames_{\blam}$, so we see that 
	\[
		\frames_{\blam}/\Sp(d) \simeq \hermitian_{\widetilde{\blam}}(N).
	\]

	In turn, if $\Delta:\hermitian_{\widetilde{\blam}}(N) \to \R^N$ is the map which records diagonal entries, then the level set $\Delta^{-1}(\br)$ is the collection of $\Sp(d)$-equivalence classes of frames with frame spectrum $\blam$ and squared frame norms $\br$; that is, 
	\[
		\Delta^{-1}(\br) \simeq \frames_{\blam}(\br)/\Sp(d).
	\]

	This means that $\frames_{\blam}(\br)$ is non-empty if and only if $\br$ is in the image of $\Delta$. But now \Cref{thm:quaternionic Schur-Horn} tells us that the image of $\Delta$ is exactly the convex hull of $S_N \cdot \widetilde{\blam}$, as desired. 
\end{proof}

Since permuting entries doesn't change their sum, $\conv(S_N \cdot \widetilde{\blam})$ lies in the affine hyperplane $\{(x_1, \dots , x_N) : \sum_{i=1}^N x_i = \sum_{i=1}^d \lambda_i\}$. Hence, $\br \in \conv(S_N \cdot \widetilde{\blam})$ only if
\begin{equation}\label{eq:admissibility equality}
	\sum_{i=1}^N r_i = \sum_{i=1}^d \lambda_i.
\end{equation}
This is just a restatement of~\eqref{eq:cyclic trace}, which said the frame operator and the Gram matrix have the same trace.

In practice, the order of the frame vectors is no more than a bookkeeping convenience, so it is no problem to permute the frame vectors. Likewise, we can freely permute the numbers comprising the frame spectrum. In particular, we can get a more straightforward criterion for non-emptiness of $\frames_{\blam}(\br)$ by sorting $\blam$ and $\br$ in non-increasing order. Specifically, if we make the additional assumption that $\lambda_1 \geq \dots \geq \lambda_d > 0$ and $r_1 \geq \dots \geq r_N > 0$, then we see that $\br \in \conv(S_N \cdot \widetilde{\blam})$ if and only if it satisfies~\eqref{eq:admissibility equality} and
\begin{equation}\label{eq:admissibility inequality}
	\sum_{i=1}^k r_i \leq \sum_{i=1}^k \lambda_i \quad \text{for all }k=1, \dots , d.
\end{equation}

We call the (sorted) $\br$ satisfying~\eqref{eq:admissibility equality} and~\eqref{eq:admissibility inequality} \emph{$\blam$-admissible}. Notice that the admissibility criterion is exactly the same as that given by Casazza and Leon in the real and complex cases~\cite{Casazza:2010ti}. Thus, \Cref{thm:admissibility} is equivalent to \Cref{thm:existence}, which we restate in a more compact form:

\begin{existence}
	Let $N$, $d$, $\blam$, and $\br$ be as above, so that $\blam$ and $\br$ are sorted in non-increasing order. Then $\frames_{\blam}(\br)$ is non-empty if and only if $\br$ is $\blam$-admissible.
\end{existence}

\section{Connectedness} 
\label{sec:connectedness}

Next, we turn to the question of when $\frames_{\blam}(\br)$ is connected. As in the previous section, we will focus on $\frames_{\blam}(\br)/\Sp(d)$, since the following lemma combined with the fact that $\Sp(d)$ is connected implies the quotient is connected if and only if $\frames_{\blam}(\br)$ is connected.

\begin{lem}\label{lem:connected quotient}
	Let $X$ be a topological space and let $G$ be a connected topological group acting continuously on $X$. If $X/G$ is connected, then $X$ is connected.
\end{lem}

This lemma follows from a standard point-set topology argument (see, e.g., \cite[Exercise~5.5]{Manetti:2015ep}) since connectedness of $G$ implies the fibers of the quotient map $X \to X/G$ are connected.

The strategy is to prove connectedness using \Cref{thm:Mare}, which applies to adjoint orbits of a compact group acting on the complementary subspace of the Lie algebra of some larger group $G$. Following the setup in \Cref{sec:isoparametric}, let $G=\SU(2N)$, so that $K=\Sp(N)$ and the complementary subspace $\mathfrak{p} \subset \mathfrak{su}(2N)$ can be identified with $\hermitian_0(N)$, the space of traceless Hermitian quaternionic $N \times N$ matrices, and the standard maximal abelian subspace of $\mathfrak{p}$ corresponds to the subset $\mathfrak{a} \subset \hermitian_0(N)$ of real diagonal $N \times N$ matrices with trace 0.

We saw in the proof of \Cref{thm:admissibility} that $\frames_{\blam}(\br)/\Sp(d) \simeq \hermitian_{\widetilde{\blam}}(N)$, the space of quaternionic Hermitian $N \times N$ matrices with spectrum $\widetilde{\blam} = (\lambda_1, \dots , \lambda_d, 0, \dots 0)$. Since $\lambda_1 + \dots + \lambda_d \neq 0$, this is not quite an orbit of the adjoint action of $\Sp(N)$ on $\hermitian_0(N)$, but it is a simple translation of such an orbit.

Letting  $\sigma:=\lambda_1 + \dots + \lambda_d$ and
\[
	\widetilde{\Lambda} := \diag\left(\lambda_1 ,\dots , \lambda_d, 0, \dots, 0\right) - \frac{\sigma}{N} \Id_N \in \hermitian_0(N),
\]
define $\tau: \hermitian_{\widetilde{\blam}}(N) \to \mathcal{O}_{\widetilde{\Lambda}}$ by $\tau(B) = B - \frac{\sigma}{N} \Id_N$.  

If $\br$ is $\blam$-admissible, then $\Delta^{-1}(\br) \subset \hermitian_{\blam}(N)$ is non-empty, and hence so is 
\[
	\tau(\Delta^{-1}(\br)) = \Delta^{-1}(t(\br)) \subset \mathcal{O}_{\widetilde{\Lambda}},
\]
where $t: \R^N \to \R^N$ is defined by $t(\boldsymbol{x}) := \boldsymbol{x} - \left(\frac{\sigma}{N},\dots , \frac{\sigma}{N}\right)$ and the above equality follows from the fact that the diagram in \Cref{lem:re-center} commutes. But now
\[
	\Delta^{-1}(t(\br)) = P^{-1}(\diag(t(\br))),
\]
where $P: \mathcal{O}_{\widetilde{\Lambda}} \to \mathfrak{a}$ is the orthogonal projection. Since $\mathcal{O}_{\widetilde{\Lambda}}$ is an isotropy orbit\footnote{In fact, it is an example of a quaternionic flag manifold, and much is known about its cohomology~\cite{Mare:2008iz,Hsiang:1988ka,Mare:2006bf}.} and hence also a parallel submanifold to some principal orbit, and since all multiplicities in this setting are equal to 4~\cite[\S 3.3]{Hsiang:1988ka}, connectedness of $\Delta^{-1}(t(\br))$ follows from \Cref{thm:Mare}. Since $\tau$ is a diffeomorphism, we conclude that 
\[
	\Delta^{-1}(\br) \simeq \frames_{\blam}(\br)/\Sp(d)
\]
is connected. 

Finally, applying \Cref{lem:connected quotient} shows that $\frames_{\blam}(\br)$ is connected whenever it is non-empty; i.e., whenever $\br$ is $\blam$-admissible. In turn, since $\frames_{\blam}(\br)$ is a real algebraic set in $\quat^{d \times N} \simeq \R^{4dN}$, it is locally path-connected (in fact, triangulable by \L{}ojasiewicz's triangulation theorem \cite{lojasiewicz1964triangulation}) so that connectivity implies path-connectivity. 

Thus, we have proved \Cref{thm:connected}, which we now restate:

\begin{connectedness}
	For any $\br$ and $\blam$, the space $\frames_{\blam}(\br)$ is path-connected.
\end{connectedness}

If we prefer to focus on the space
\[
	\frames_S(\br) := \{F \in \frames : FF^\ast = S \text{ and } \|f_i\|^2 = r_i \text{ for all } i = 1 , \dots , N\}
\]
of frames with fixed frame operator (rather than fixed frame spectrum) and fixed frame vector norms, connectivity still holds:

\begin{cor}\label{cor:fixed frame operator}
	If $S \in \hermitian(N)$ is positive definite with spectrum $\blam$ and $\br \in \R_+^N$, the space $\frames_S(\br)$ is (i) non-empty if and only if $\br$ is $\blam$-admissible and (ii) path-connected.
\end{cor}

\begin{proof}
	If $\frames_S(\br) \subset \frames_{\blam}(\br)$ is non-empty, then certainly $\br$ is $\blam$-admissible by \Cref{thm:admissibility}. Conversely, if $\br$ is $\blam$-admissible, then $\frames_{\blam}(\br)$ is non-empty, so that there exists $F \in \frames_{\blam}(\br)$, and $FF^\ast = V S V^\ast$ for some $V \in \Sp(d)$. But then $V^\ast F \in \frames_S(\br)$, so we see that $\frames_S(\br)$ is non-empty as well.
	
	Since the empty set is trivially path-connected, the only thing to prove is that $\frames_S(\br)$ is path-connected whenever it is non-empty. In this case, we know from \Cref{thm:connected} that $\frames_{\blam}(\br)$ is path-connected. So if $F_0, F_1 \in \frames_S(\br) \subset \frames_{\blam}(\br)$, then there exists a path $\widetilde{\gamma}: [0,1] \to \frames_{\blam}(\br)$ so that $\widetilde{\gamma}(0) = F_0$ and $\widetilde{\gamma}(1) = F_1$.
	
	Of course, the path $\widetilde{\gamma}$ need not stay in $\frames_S(\br)$, but we can easily fix it up to do so. Indeed, the frame operator
	\[
		\widetilde{\gamma}(t)\widetilde{\gamma}(t)^\ast = U_t S U_t^\ast
	\]
	determines a continuous path $\{U_t: t \in [0,1]\} \subset \Sp(d)$, and hence
	\[
		\gamma(t) := U_t^\ast \widetilde{\gamma}(t)
	\]
	is a path in $\frames_S(\br)$ connecting $F_0$ and $F_1$.
\end{proof}

\section{Discussion}
\Cref{cor:fixed frame operator} is the exact analog of the main theorem in our previous paper~\cite{NeedhamSGC} on complex frames. While our argument in the complex case was based on symplectic geometry, the strategy employed here for proving \Cref{thm:connected} and \Cref{cor:fixed frame operator} can be adapted to the complex setting, so we view this approach as somewhat more general.

Indeed, much of the setup goes through in the real case as well: orthogonal equivalence classes of real frames with fixed frame spectrum correspond to adjoint orbits by way of their Gram matrices, and hence are parallel to isoparametric submanifolds of the space of symmetric matrices. Unfortunately, all multiplicities in this case are equal to 1~\cite[\S 3.2]{Hsiang:1988ka}, so \Cref{thm:Mare} does not apply. Indeed, while some of the real frame spaces $\F^{\R^d,N}_{\blam}(\br)$ are connected~\cite{Cahill:2017gv}, others are not~\cite{Kapovich:1995wg} (compare with~\cite[Theorem~2.7]{Goyal:2001cd}), and it seems challenging (but interesting!) to characterize which $\blam$ and $\br$ lead to which outcome. Since the language of isoparametric submanifolds provides a common framework for understanding real, complex, and quaternionic frames, this perspective seems ripe for further exploration.

\Cref{thm:existence,thm:connected} have the same statements as the corresponding results in the complex case, which gives some reason to hope that direct translations of other results about complex frames might give true statements about quaternionic frames. For example, it seems likely that there is a quaternionic extension of the eigenstep method~\cite{Cahill:2013jv} which would construct all elements of $\frames_{\blam}(\br)$, and more generally give a constructive proof of \Cref{thm:quaternionic Schur-Horn} (cf.~\cite{Fickus:2013ki}).

\subsection*{Acknowledgments}

We are very grateful to Augustin-Liviu Mare, Emily King, and Colin Roberts for providing inspiration and sharing their knowledge and insight. This work was partially supported by grants from the National Science Foundation (DMS--2107808, Tom Needham; DMS--2107700, Clayton Shonkwiler).

\bibliography{needham_bibliography,shonkwiler-papers}

\begin{thebibliography}{10}

\bibitem{Antezana:2007ci}
Jorge Antezana, Pedro~G Massey, Mariano~A Ruiz, and Demetrio Stojanoff.
\newblock {The Schur--Horn theorem for operators and frames with prescribed
  norms and frame operator}.
\newblock {\em Illinois Journal of Mathematics}, 51(2):537--560, 2007.

\bibitem{Atiyah:1982ih}
Michael~Francis Atiyah.
\newblock {Convexity and commuting Hamiltonians}.
\newblock {\em Bulletin of the London Mathematical Society}, 14(1):1--15, 1982.

\bibitem{Audin:2004bh}
Mich{\`e}le Audin.
\newblock {\em {Torus Actions on Symplectic Manifolds}}, volume~93 of {\em
  Progress in Mathematics}.
\newblock Birkh\"auser Verlag, Basel, second revised edition, 2004.

\bibitem{SPspecialissue}
Nicolas~Le Bihan, Todd~A Ell, Danilo Mandic, Tohru Nitta, and {Stephen J
  Sangwine, editors}.
\newblock {Hypercomplex Signal Processing [special issue]}.
\newblock \emph{Signal Processing} 136, 2017.

\bibitem{Cahill:2013jv}
Jameson Cahill, Matthew Fickus, Dustin~G Mixon, Miriam~J Poteet, and Nate
  Strawn.
\newblock {Constructing finite frames of a given spectrum and set of lengths}.
\newblock {\em Applied and Computational Harmonic Analysis}, 35(1):52--73,
  2013.

\bibitem{Cahill:2017gv}
Jameson Cahill, Dustin~G Mixon, and Nate Strawn.
\newblock {Connectivity and irreducibility of algebraic varieties of finite
  unit norm tight frames}.
\newblock {\em SIAM Journal on Applied Algebra and Geometry}, 1(1):38--72,
  2017.

\bibitem{Cantarella:2013bla}
Jason Cantarella, Tetsuo Deguchi, and Clayton Shonkwiler.
\newblock {Probability theory of random polygons from the quaternionic
  viewpoint}.
\newblock {\em Communications on Pure and Applied Mathematics},
  67(10):1658--1699, 2014.

\bibitem{Cartan:1938kh}
{\'E}lie Cartan.
\newblock {Familles de surfaces isoparam{\'e}triques dans les espaces {\`a}
  courbure constante}.
\newblock {\em Annali di Matematica Pura ed Applicata. Serie Quarta},
  17:177--191, 1938.

\bibitem{Cartan:1939gv}
{\'E}lie Cartan.
\newblock {Sur des familles remarquables d'hypersurfaces isoparam{\'e}triques
  dans les espaces sph{\'e}riques}.
\newblock {\em Mathematische Zeitschrift}, 45:335--367, 1939.

\bibitem{Cartan:1939wj}
{\'E}lie Cartan.
\newblock {Sur quelques familles remarquables d'hypersurfaces}.
\newblock In {\em Comptes rendus du Congr{\`e}s des Sciences Mathematiques de
  Li{\`e}ge (17-22 juillet 1939)}, pages 30--41. Georges Thone, Li{\`e}ge,
  1939.

\bibitem{Cartan:tw}
{\'E}lie Cartan.
\newblock {Sur des familles d'hypersurfaces isoparametriques des espaces
  spheriques a 5 et a 9 dimension}.
\newblock {\em Universidad Nacional de Tucum{\'a}n. Facultad de Ciencias
  Exactas y Tecnolog{\'\i}a. Revista. Serie A. Matem{\'a}tica y F{\'\i}sica
  Te{\'o}rica}, 1:5--22, 1940.

\bibitem{Casazza:2006cx}
Peter~G Casazza, Matthew Fickus, Jelena Kova{\v c}evi{\'c}, Manuel~T Leon, and
  Janet~C Tremain.
\newblock {A physical interpretation of tight frames}.
\newblock In Christopher Heil, editor, {\em Harmonic Analysis and Applications:
  In Honor of John J. Benedetto}, Applied and Numerical Harmonic Analysis,
  pages 51--76. Birkh\"auser, Boston, 2006.

\bibitem{Casazza:2011ev}
Peter~G Casazza, Andreas Heinecke, Felix Krahmer, and Gitta Kutyniok.
\newblock {Optimally sparse frames}.
\newblock {\em IEEE Transactions on Information Theory}, 57(11):7279--7287,
  2011.

\bibitem{Casazza:2003vp}
Peter~G Casazza and Jelena Kova{\v c}evi{\'c}.
\newblock {Equal-norm tight frames with erasures}.
\newblock {\em Advances in Computational Mathematics}, 18(2--4):387--430, 2003.

\bibitem{Casazza:2010ti}
Peter~G Casazza and Manuel~T Leon.
\newblock {Existence and construction of finite frames with a given frame
  operator}.
\newblock {\em International Journal of Pure and Applied Mathematics},
  63(2):149--157, 2010.

\bibitem{Cohn:2016bz}
Henry Cohn, Abhinav Kumar, and Gregory Minton.
\newblock {Optimal simplices and codes in projective spaces}.
\newblock {\em Geometry and Topology}, 20(3):1289--1357, 2016.

\bibitem{Dykema:2006ux}
Ken Dykema and Nate Strawn.
\newblock {Manifold structure of spaces of spherical tight frames}.
\newblock {\em International Journal of Pure and Applied Mathematics},
  28(2):217--256, 2006.

\bibitem{ell2006hypercomplex}
Todd~A Ell and Stephen~J Sangwine.
\newblock Hypercomplex fourier transforms of color images.
\newblock {\em IEEE Transactions on Image Processing}, 16(1):22--35, 2006.

\bibitem{EtTaoui:2020kf}
Boumediene Et-Taoui.
\newblock {Quaternionic equiangular lines}.
\newblock {\em Advances in Geometry}, 20(2):273--284, 2020.

\bibitem{Fickus:2013ki}
Matthew Fickus, Dustin~G Mixon, Miriam~J Poteet, and Nate Strawn.
\newblock {Constructing all self-adjoint matrices with prescribed spectrum and
  diagonal}.
\newblock {\em Advances in Computational Mathematics}, 39(3--4):585--609, 2013.

\bibitem{fletcher2017development}
Peter Fletcher and Stephen~J Sangwine.
\newblock The development of the quaternion wavelet transform.
\newblock {\em Signal Processing}, 136:2--15, 2017.

\bibitem{Goyal:2001cd}
Vivek~K Goyal, Jelena Kova{\v c}evi{\'c}, and Jonathan~A Kelner.
\newblock {Quantized frame expansions with erasures}.
\newblock {\em Applied and Computational Harmonic Analysis}, 10(3):203--233,
  2001.

\bibitem{Guillemin:1982gx}
Victor Guillemin and Shlomo Sternberg.
\newblock {Convexity properties of the moment mapping}.
\newblock {\em Inventiones Mathematicae}, 67(3):491--513, 1982.

\bibitem{Hanson:2006tr}
Andrew~J Hanson.
\newblock {\em {Visualizing Quaternions}}.
\newblock Morgan Kaufmann, San Francisco, 2006.

\bibitem{Hanson:2012vb}
Andrew~J Hanson and Sidharth Thakur.
\newblock {Quaternion maps of global protein structure}.
\newblock {\em Journal of Molecular Graphics and Modelling}, 38:256--278, 2012.

\bibitem{Helgason:1978vb}
Sigurdur Helgason.
\newblock {\em {Differential Geometry, Lie Groups, and Symmetric Spaces}},
  volume~80 of {\em Pure and Applied Mathematics}.
\newblock Academic Press, New York, 1978.

\bibitem{Holmes:2004iv}
Roderick~B Holmes and Vern~I Paulsen.
\newblock {Optimal frames for erasures}.
\newblock {\em Linear Algebra and its Applications}, 377:31--51, 2004.

\bibitem{horn1954doubly}
Alfred Horn.
\newblock Doubly stochastic matrices and the diagonal of a rotation matrix.
\newblock {\em American Journal of Mathematics}, 76(3):620--630, 1954.

\bibitem{Howard:2011fj}
Benjamin~J Howard, Christopher Manon, and John~J Millson.
\newblock {The toric geometry of triangulated polygons in Euclidean space}.
\newblock {\em Canadian Journal of Mathematics}, 63(4):878--937, 2011.

\bibitem{Hsiang:1988ka}
Wu-Yi Hsiang, Richard~S Palais, and Chuu-Lian Terng.
\newblock {The topology of isoparametric submanifolds}.
\newblock {\em Journal of Differential Geometry}, 27(3):423--460, 1988.

\bibitem{iverson_note_2021}
Joseph~W Iverson, Emily~J King, and Dustin~G Mixon.
\newblock A note on tight projective 2-designs.
\newblock {\em Journal of Combinatorial Designs}, 29(12):809--832, 2021.

\bibitem{Kapovich:1995wg}
Michael Kapovich and John~J Millson.
\newblock {On the moduli space of polygons in the Euclidean plane}.
\newblock {\em Journal of Differential Geometry}, 42(2):430--464, 1995.

\bibitem{Khokulan:2017ic}
Mohananathan Khokulan, Kengatharam Thirulogasanthar, and Sivakolundu
  Srisatkunarajah.
\newblock {Discrete frames on finite dimensional left quaternion Hilbert
  spaces}.
\newblock {\em Axioms}, 6(1):3, 2017.

\bibitem{kobert_spectrahedral_2021}
Tim Kobert and Claus Scheiderer.
\newblock Spectrahedral representation of polar orbitopes.
\newblock {\em manuscripta mathematica}, 2021.
\newblock \url{https://doi.org/10.1007/s00229-021-01337-z}.

\bibitem{Kostant:1973ff}
Bertram Kostant.
\newblock {On convexity, the Weyl group and the Iwasawa decomposition}.
\newblock {\em Annales Scientifiques de l'{\'E}cole Normale Sup{\'e}rieure,
  S{\'e}rie 4}, 6(4):413--455, 1973.

\bibitem{Kovacevic:2007fja}
Jelena Kova{\v c}evi{\'c} and Amina Chebira.
\newblock {Life Beyond Bases: The Advent of Frames (Part I)}.
\newblock {\em IEEE Signal Processing Magazine}, 24(4):86--104, 2007.

\bibitem{lojasiewicz1964triangulation}
Stanis\l{}aw \L{}ojasiewicz.
\newblock Triangulation of semi-analytic sets.
\newblock {\em Annali della Scuola Normale Superiore di Pisa-Classe di
  Scienze}, 18(4):449--474, 1964.

\bibitem{Manetti:2015ep}
Marco Manetti.
\newblock {\em {Topology}}, volume~91 of {\em Unitext}.
\newblock Springer, Cham, 2015.

\bibitem{Mare:2005eo}
Augustin-Liviu Mare.
\newblock {Connectivity and Kirwan surjectivity for isoparametric
  submanifolds}.
\newblock {\em International Mathematics Research Notices},
  2005(55):3427--3443, 2005.

\bibitem{Mare:2006bf}
Augustin-Liviu Mare.
\newblock {Equivariant cohomology of real flag manifolds}.
\newblock {\em Differential Geometry and its Applications}, 24(3):223--229,
  2006.

\bibitem{Mare:2008iz}
Augustin-Liviu Mare.
\newblock {Equivariant cohomology of quaternionic flag manifolds}.
\newblock {\em Journal of Algebra}, 319(7):2830--2844, 2008.

\bibitem{mcduff2017introduction}
Dusa McDuff and Dietmar Salamon.
\newblock {\em Introduction to Symplectic Topology}.
\newblock Oxford University Press, third edition, 2017.

\bibitem{Needham:2017vd}
Tom Needham.
\newblock {Knot types of generalized Kirchhoff rods}.
\newblock {\em Journal of Knot Theory and Its Ramifications}, 28(11):1940010,
  2019.

\bibitem{NeedhamSGC}
Tom Needham and Clayton Shonkwiler.
\newblock {Symplectic geometry and connectivity of spaces of frames}.
\newblock {\em Advances in Computational Mathematics}, 47(1):5, 2021.

\bibitem{Nomizu:1975bq}
Katsumi Nomizu.
\newblock {Elie Cartan's work on isoparametric families of hypersurfaces}.
\newblock In Shiing-Shen Chern and Robert Osserman, editors, {\em Differential
  Geometry}, volume 27, part 1 of {\em Proceedings of Symposia in Pure
  Mathematics}, pages 191--200. American Mathematical Society, Providence,
  1975.

\bibitem{Palais:1988ks}
Richard~S Palais and Chuu-Lian Terng.
\newblock {\em {Critical Point Theory and Submanifold Geometry}}, volume 1353
  of {\em Lecture Notes in Mathematics}.
\newblock Springer-Verlag, Berlin, 1988.

\bibitem{Rodman:2014ue}
Leiba Rodman.
\newblock {\em {Topics in Quaternion Linear Algebra}}, volume~45 of {\em
  Princeton Series in Applied Mathematics}.
\newblock Princeton University Press, Princeton, 2014.

\bibitem{Rupf:1994fl}
Marcel Rupf and James~L Massey.
\newblock {Optimum sequence multisets for synchronous code-division
  multiple-access channels}.
\newblock {\em IEEE Transactions on Information Theory}, 40(4):1261--1266,
  1994.

\bibitem{schur1923uber}
Issai Schur.
\newblock Uber eine {K}lasse von {M}ittelbildungen mit {A}nwendungen auf die
  {D}eterminantentheorie.
\newblock {\em Sitzungsberichte der Berliner Mathematischen Gesellschaft},
  22:9--20, 1923.

\bibitem{Sharma:2019js}
Sumit~Kumar Sharma and Shashank Goel.
\newblock {Frames in quaternionic Hilbert space}.
\newblock {\em \foreignlanguage{russian}{Журнал
  математической физики, анализа,
  геометрии} [Journal of Mathematical Physics, Analysis, Geometry]},
  15(3):395--411, 2019.

\bibitem{Terng:1985eu}
Chuu-Lian Terng.
\newblock {Isoparametric submanifolds and their Coxeter groups}.
\newblock {\em Journal of Differential Geometry}, 21(1):79--107, 1985.

\bibitem{Terng:1986hg}
Chuu-Lian Terng.
\newblock {Convexity theorem for isoparametric submanifolds}.
\newblock {\em Inventiones Mathematicae}, 85(3):487--492, 1986.

\bibitem{Tropp:2005do}
Joel~A Tropp, Inderjit~S Dhillon, Robert~W Heath~Jr, and Thomas Strohmer.
\newblock {Designing structured tight frames via an alternating projection
  method}.
\newblock {\em IEEE Transactions on Information Theory}, 51(1):188--209, 2005.

\bibitem{Virender:2020ua}
{Virender}, Sumit~Kumar Sharma, Ghanshyam Singh, and Soniya Sahu.
\newblock {On frames in finite dimensional quaternionic Hilbert space}.
\newblock {\em Palestine Journal of Mathematics}, 9(1):511--522, 2020.

\bibitem{Viswanath:1999hf}
Pramod Viswanath and Venkat Anantharam.
\newblock {Optimal sequences and sum capacity of synchronous CDMA systems}.
\newblock {\em IEEE Transactions on Information Theory}, 45(6):1984--1991,
  1999.

\bibitem{Viswanath:2002bv}
Pramod Viswanath and Venkat Anantharam.
\newblock {Optimal sequences for CDMA under colored noise: a Schur-saddle
  function property}.
\newblock {\em IEEE Transactions on Information Theory}, 48(6):1295--1318,
  2002.

\bibitem{Waldron:2020tp}
Shayne F~D Waldron.
\newblock {A variational characterisation of projective spherical designs over
  the quaternions}.
\newblock Preprint, {\tt arXiv:2011.08439 [cs.IT]}, 2020.

\bibitem{Waldron:2020ti}
Shayne F~D Waldron.
\newblock {Tight frames over the quaternions and equiangular lines}.
\newblock Preprint, {\tt arXiv:2006.06126 [math.FA]}, 2020.

\bibitem{yi2017quaternion}
Cancan Yi, Yong Lv, Zhang Dang, Han Xiao, and Xun Yu.
\newblock Quaternion singular spectrum analysis using convex optimization and
  its application to fault diagnosis of rolling bearing.
\newblock {\em Measurement}, 103:321--332, 2017.

\bibitem{Zhang:1997cd}
Fuzhen Zhang.
\newblock {Quaternions and matrices of quaternions}.
\newblock {\em Linear Algebra and its Applications}, 251:21--57, 1997.

\bibitem{zhao2020quaternion}
Qiang Zhao, Qizhen Du, Qamar Yasin, Qingqing Li, and Liyun Fu.
\newblock Quaternion-based sparse tight frame for multicomponent signal
  recovery.
\newblock {\em Geophysics}, 85(2):V143--V156, 2020.

\end{thebibliography}

\end{document}